\documentclass{amsart}
\usepackage[utf8]{inputenc}
\usepackage{amsmath}
\usepackage{amssymb}
\usepackage{parskip}
\usepackage{mathtools}
\usepackage{enumitem,xcolor}
\usepackage{amsthm}
\usepackage[margin=2.5cm]{geometry}
\usepackage{hyperref}
\usepackage{mathrsfs}
\usepackage{graphicx}
\usepackage{multicol}
\usepackage{stmaryrd}
\usepackage{tikz-cd}
\usetikzlibrary{matrix,arrows,decorations.pathmorphing}
\usepackage{scalerel,stackengine}
\usepackage{xparse}
\usepackage[style=numeric,giveninits]{biblatex}
\makeatletter
\newenvironment{smallermatrix}[1][c]
{\null\,\vcenter\bgroup
  \Let@\restore@math@cr\default@tag
  \baselineskip0pt \lineskip0.4pt \lineskiplimit0pt
  \ialign\bgroup\if#1l\else\hfil\fi$\m@th\scriptstyle##$\if#1r\else\hfil\fi&&\thickspace\hfil
  $\m@th\scriptstyle##$\hfil\crcr
}{%
  \crcr\egroup\egroup\,%
}
\makeatother

\NewDocumentCommand{\ts}{O{c} e{^?_}}{
  \begin{smallermatrix}[#1]
  \mathstrut\IfValueT{#2}{#2} \\
  \mathstrut\IfValueT{#3}{#3} \\
  \mathstrut\IfValueT{#4}{#4}
  \end{smallermatrix}%
}

\newtheorem{thm}{Theorem}
\newtheorem{prop}[thm]{Proposition}

\newtheorem{definition}[thm]{Definition}

\stackMath
\newcommand\reallywidehat[1]{%
\savestack{\tmpbox}{\stretchto{%
  \scaleto{%
    \scalerel*[\widthof{\ensuremath{#1}}]{\kern-.6pt\bigwedge\kern-.6pt}%
    {\rule[-\textheight/2]{1ex}{\textheight}}
  }{\textheight}%
}{0.5ex}}%
\stackon[1pt]{#1}{\tmpbox}%
}

\DeclareFontFamily{U}{wncy}{}
    \DeclareFontShape{U}{wncy}{m}{n}{<->wncyr10}{}
    \DeclareSymbolFont{mcy}{U}{wncy}{m}{n}
    \DeclareMathSymbol{\Sh}{\mathord}{mcy}{"58}

\begin{document}
\title{The de Rham period map for punctured elliptic curves and the KZB equation}
\author{Ben Moore}
\address{Mathematics Institute, University of Warwick, CV4 7AL, UK, and School of Mathematics,
University of Edinburgh, King’s Buildings, Edinburgh, UK}
\date{May 2023}

\maketitle

\begin{abstract}
    We demonstrate that the algebraic KZB connection of Levin--Racinet and Luo on a once-punctured elliptic curve represents Kim's universal unipotent connection, and we observe that the Hodge filtration on the KZB connection has a particularly simple form. This allows us to generalise previous work of Beacom by writing down explicitly the maximal metabelian quotient of Kim's de Rham period map in terms of elliptic polylogarithms. As far as we are aware this is the first time that the de Rham period map has been written out for an infinite dimensional quotient of the de Rham fundamental group on any curve of positive genus.
\end{abstract}

\section{Introduction}
Let $C/K$ be a geometrically smooth complete curve of genus $g$ over a local field $K$ with good reduction over the residue field $k$. Write $X$ for the hyperbolic curve $C-D$, where $D$ is a finite subset of cardinality $r\geq 3-2g$. As part of his program \cite{motivicfundP1, massey, masseyerratum, bdckw} for understanding hyperbolic curves over number fields via \emph{Selmer schemes}, Kim has defined a \emph{de Rham period map} which associates to each $K$-valued point of $X$ an element inside a certain moduli space of torsors over $X$ (here, $K$ is the completion of the number field at some finite place of good reduction). The entire Selmer scheme apparatus is controlled by the $K$-prounipotent completion of the motivic fundamental group $U$ of $X$, which is an infinite dimensional $K$-prounipotent group, and the de Rham period map is controlled by the de Rham realisation $U^{dR}$ of $U$.

Previous work \cite{quadchab, quadchabI, split13} has generally focussed on the Selmer schemes induced by certain small finite-dimensional quotients of $U$, such as its descending central series quotients $U_n$ for $n\leq 3$, but such quotients contain nontrivial information about $X$ only when $n$ exceeds some (increasing) function of the rank of the Jacobian of $X$. It is therefore necessary (and natural) to try to understand Selmer schemes for $U$ itself --- or more realistically, for algebraically well-behaved infinite-dimensional quotients of $U$, such as its maximal metabelianisation $W=U/[[U,U],[U,U]]$. In this article we give a complete description of the de Rham period map for $W^{dR}$ in the simplest nontrivial case, namely for $X$ a once-punctured elliptic curve.

The target of the de Rham period map is the affine space $F^0U^{dR}\backslash U^{dR}$, where $F^{\bullet}$ is the \emph{Hodge filtration} on $U^{dR}$. In general, we have $F^0U^{dR}=\exp\left(F^0 \mathfrak{u}\right)$, where $\mathfrak{u}$ is the Lie algebra of $U^{dR}$ and $F^{0}\mathfrak{u}$ is a subspace. For the purposes of applications to Selmer schemes, it is necessary to understand the composition of the period map with a coordinate $\iota: F^0U\backslash U\rightarrow \mathfrak{u}$. In general, one obtains such a coordinate $\iota$ from a splitting $\mathfrak{u}=F^{0}\mathfrak{u}\oplus M$ of $K$-vector spaces, with $M$ a Lie ideal, simply by observing that the composite \begin{equation*}
    M\xrightarrow{\mathrm{exp}}U \twoheadrightarrow F^0 U\backslash U
\end{equation*} is bijective, and defining $\iota$ to be its inverse.

When $X$ is the thrice-punctured line, $F^0$ is trivial, and so $\iota$ is simply the logarithm. It is easy to compute the image in $W$ of the logarithm of the de Rham period map; however, most authors choose to compute its image in Deligne's polylogarithmic quotient \cite{deligneP1}, which is even smaller than $W$. For a once-punctured elliptic curve $X=E-\{\infty\}$, $\mathfrak{u}$ is the free Lie algebra on $A$ and $B$ ($A$ corresponds to an algebraic $1$-form on $E$ and $B$ to an algebraic $1$-form on $X$). Furthermore $F^0\mathfrak{u}$ is one-dimensional. Choosing a Lie ideal $M$ allows us to define the projection $\pi_{F^0\mathfrak{u}}$ of $\mathfrak{u}$ onto $F^0\mathfrak{u}$, and we can check that the map \begin{equation*}
U \ni x\mapsto\log\left(\exp\left(-\pi_{F^0\mathfrak{u}}(\log(x))\right)x)\right) \in M
\end{equation*} descends to $F^0U\backslash U$ and is equal to $\iota$. It is natural to wonder whether our results are contained in the earlier article of Beacom \cite{beacom}, which is dedicated to the de Rham period map for elliptic and hyperelliptic curves. The answer is that Beacom describes a recursive algorithm for computing the de Rham period map attached to $U_n$ from the one attached to $U_{n-1}$, but the algorithm is sufficiently complicated that it is infeasible to use it to directly write down the period map for every $U_n$.

We circumvent this difficulty by exploiting the fact, which seems to be generally under-appreciated, that the presentation of $F^0 U^{dR}$, and hence $\iota$, depends on a choice of representative for the \emph{universal connection} on $X$. Beacom, along with all other writers on the subject, chooses the so-called \emph{naive connection} $\nabla_{0}$, for which the \emph{crystalline path torsor} appearing in the period map assumes a particularly simple form. But the drawback to this approach is that the Hodge filtration is harder to describe: for example, at level four, we have \begin{equation*}
F_{\nabla_0}^0\mathfrak{u}_4=\mathrm{Span}_{K}\left\{B-\lambda[B,[A,B]]+(\mu+\kappa/3)[B,[B,[A,B]]]\right\}, \quad \lambda, \mu, \kappa \in K.
\end{equation*} However, we choose to employ (a variant of) the algebraic KZB connection $\nabla_{KZB}$ of Levin--Racinet and Luo. There is a unique (pointed) isomorphism between these (pointed) connections, but this variant of the KZB connection enjoys a number of highly desirable properties. Among them is the particularly simple form taken by its the Hodge filtration: \begin{equation*}
F_{\nabla_{KZB}}^0\mathfrak{u}=\mathrm{Span}_{K}\{B\}.
\end{equation*} Now there is a canonical choice for $M$: namely, the Lie ideal spanned by Lie words of degree at least $1$ in $A$. Then $\iota$ takes the form \begin{equation}\label{eq:coordonperiod}
    \iota_{KZB}(x)=\log\left(\exp\left(-\pi_{\mathrm{Span}_{K}\{B\}}(\log(x))\right)x)\right).
\end{equation} Finally, we note that as for the thrice-punctured line, it is very convenient to pass to $W$, in order to make the calculations algebraically tractable. We require descriptions of the flat section of the (adjoint) of the KZB equation, of the logarithm of this flat section, and of the Baker--Campbell--Hausdorff formula, and all of these are rather difficult to obtain over $U$.

Before we state the main result of the article, we need to introduce some algebraic functions on our punctured elliptic curve. Let $E/K$ be an elliptic curve over a local field $K$ of characteristic zero, with good reduction over the residue field $K$. Let $X=E-\{\infty\}$ denote the hyperbolic curve, with model \begin{equation*}
        y^2=4x^3-60e_4 x-140 e_6
    \end{equation*} obtained by removing the point at infinity and let $b \in E(K)$ be a basepoint. The Lie algebra $\mathfrak{w}$ of the maximal metabelian quotient of the de Rham fundamental group of $X$ based at $b$ is spanned by $A$, $B$ and $\sigma_{r,s}=\mathrm{ad}^{r}_{B}\mathrm{ad}^{s}_{A}[A,B]$ for $r,s\geq 0$ and we can choose these so that the isomorphism $\mathfrak{w}^{\mathrm{ab}}\cong H^{dR}_1(X/K)$ identifies the dual of $A$ with the global regular differential $\alpha=dx/y$ and the dual of $B$ with the differential form $\beta=xdx/y$ which is regular away from $\infty$.  Let $e_{n}$ be $1/(2n+1)$ times the $(n-2)$th Laurent series coefficient of the Weierstrass function for $E$: alternatively, the $e_n$ are characterised by $e_n=0$ for $n\geq 3 $ odd and \begin{equation*}
\frac{1}{3}(m-3)(4m^2-1)e_{2m}=\sum_{r=2}^{m-2}(2r-1)(2m-r-1)e_{2r}e_{2m-2r}
    \end{equation*} for $m\geq 4$. Then let $P_1=-2x^2/y$, $P_2=x$, $P_3=-y/2$ and for $m,n \geq 2$, define $P_m$ by \begin{equation*}
        P_m P_n-P_{m+n}=(-1)^n \sum_{k=1}^{m-2}\binom{n+k-1}{k}e_{n+k}P_{m-k}+(-1)^m \sum_{k=1}^{n-2}\binom{m+k-1}{k}e_{m+k}P_{n-k}+(-1)^m\binom{m+n}{m}e_{m+n}.
    \end{equation*} We recall the Bernoulli numbers $B_{m}$ with the convention \begin{equation}\label{eq:bernoulli}
        B_{0}=-1, B_{1}=\frac{1}{2},B_{2}=\frac{1}{6},\dots.
    \end{equation}
    \begin{thm}\label{thm:periodmap}
     We have \begin{equation}\label{eq:periodeq}
\iota_{KZB}\circ j^{dR}(X,b)(x)=\int_{b}^{x}\alpha A+\sum_{r,s\geq 0}g_{r,s}\sigma_{r,s}-\int_{b}^{x}\beta S\left(\int_{b}^{x}\alpha\mathrm{ad}_{A}, \int_{b}^{x}\beta\mathrm{ad}_{B}\right)\left(-\int_{b}^{x}\alpha \sigma_{0,0}+\sum_{r,s\geq 0}g_{r,s}\sigma_{r+1,s}\right)
   \end{equation} where $j^{dR}(X,b)(x)$ is the de Rham period map, $\iota_{KZB}$ is the image inside $\mathfrak{w}$ of the coordinate described above, and \begin{equation*}
       S(U,V)=\frac{1}{U+V}\left(1-\frac{\exp(-V)-1}{V}\frac{U+V}{\exp(U+V)-1}\right).
   \end{equation*} The $g_{r,s}$ are iterated integrals of certain algebraic functions on $E/K$, characterised by the equality of formal power series \begin{equation*}
       \sum_{r,s\geq 0}g_{r,s} X^s Y^r=\sum_{r,s\geq 0}\left(\sum_{\substack{u+n=r\\v+m=s}}\frac{B_{u+v}\left(\int_{b}^{x}\beta\right)^u\left(\int_{b}^{x}\alpha\right)^v G_{n,m}}{u!v!}\right)X^s Y^r-\left(\int_{b}^{x}\alpha\right) \left(\int_{b}^{x}\beta\right) T\left(\int_{b}^{x}\alpha X, \int_{b}^{x}\beta Y\right),
   \end{equation*} where \begin{equation*}
       T(U,V)=\frac{1}{U}\left(1-\frac{U+V}{V}\frac{\exp(U)-1}{\exp(U+V)-1}\right)
   \end{equation*} and $B_{m}$ is the $m$th Bernoulli number from (\ref{eq:bernoulli}). The algebraic iterated integrals $G_{r,s}$ are in turn characterised as the iterated integral of the coefficient of $X^r Y^s$ in 
\begin{equation*}
     -\left(1-\alpha Y-\beta X\right)^{-1}\left(1+(1-\beta X)\sum_{i\geq 0}\sum_{k=2}^{i}\beta^{i-k+1}\alpha p_k X^i\right),
\end{equation*} where we define the algebraic functions $p_k$ by \begin{equation*}
p_n=\sum_{2a_2+3a_3+\dots+na_n=n}\frac{1}{a_2! a_3!\dots a_n!}\prod_{k=2}^{n}\left(\frac{(-1)^{k+1}P_k}{k}\right)^{a_k}.
    \end{equation*}
\end{thm}
In Theorem \ref{thm:tangperiodmap}, we prove a similar result in the case that $b$ is a tangential basepoint at $\infty$.

We note, in passing, that our KZB connection is a connection on a bundle of cocommutative Hopf algebras, whereas the KZB connections used by Levin--Racinet \cite{levinracinet}, Hain \cite{hainkzbnotes}, Luo \cite{luo}, Brown--Levin \cite{brownlevin} and others are closely related connections on a bundle of Lie algebras. This latter KZB connection is important primarily because it is the specialisation to $E$ of the unique logarithmic connection over the moduli stack of elliptic curves satisfying certain properties \cite[Corollary 14.4]{hainkzbnotes}. With respect to higher genus curves, Enriquez has defined a KZB connection on Schottky covers of configuration spaces of complex curves \cite{enriquezkzb}, and Enriquez--Zerbini have given (in principle) gauge transformations under which this connection descends to the curve \cite{enriquezzerbini}, but at the time of writing these constructions haven't been worked out explicitly for any curve of genus at least $2$. In fact it is not known if their gauge transformations yield an algebraic connection.

However, for the purposes of writing down the de Rham period map on hyperbolic curves $X=C-P$, where $C$ is a complete curve of genus $g\geq 2$ and $P \in C$, with respect to the maximal metabelian quotient $\mathfrak{w}$ of $\mathfrak{u}$, it suffices to construct any connection of a certain form, which ensures that generalisations of Propositions \ref{prop:kzbuniversal} and \ref{prop:hodgefilt} hold. Then, straightforward generalisations of Propositions \ref{prop:flatad} and \ref{prop:logarithm} to the maximal metabelian Lie algebra on $n=2g$ generators $A_1,B_1\dots,A_g,B_g$, \begin{equation*}
\mathfrak{w}_{n}=\mathrm{Span}_{K}\left\{X_1,\dots X_n,\sigma_{\mathbf{u}}\rvert_{i,j}\right\}_{1\leq i<j\leq n},\quad \mathbf{u}=(0,\dots,0,u_i,\dots, u_n),\quad \sigma_{\mathbf{u}}\rvert_{i,j}=\mathrm{ad}_{X_{n}}^{u_n}\dots\mathrm{ad}_{X_{1}}^{u_1}[X_i,X_j],
\end{equation*} would yield analogous formulae for the period map. Specifically, cover $C$ by two open subsets $X=C-P$ and $Y$: then we may construct our connection by specifying two $\mathfrak{w}$-valued one-forms $\omega_X$ and $\omega_Y$ on $X$ and $Y$ together with a gauge transformation $g$ such that $\omega_Y=-dg.g^{-1}+g\omega g^{-1}$ on $X\cap Y$. Then the obvious analogues of Propositions \ref{prop:kzbuniversal} and \ref{prop:hodgefilt} hold if:
\begin{enumerate}
    \item $\omega_X=\sum_{i=1}^{2g}\left(\alpha_iA_i+\beta_i B_i\right)$ modulo higher degree words (the $\alpha_i$ and $\beta_i$ are a basis for $H^1_{dR}(X/K)$ such that the $\alpha_i$ extend to global $1$-forms on $X$);
    \item $\omega_X$ and $\omega_Y$ take values in the (not necessarily Galois-equivariant) Lie subalgebra $\mathfrak{g}$ on $\mathfrak{u}$ consisting of Lie words of degree at most $1$ in $A_1,\dots,A_g$; and
    \item $g$ acts as left multiplication by words in the $B_i$s.
\end{enumerate}
The period map for the completion of $C$ is obtained by quotienting by the relation $[A_1,B_1]+\dots+[A_g,B_g]=0$.
\section{Definitions}
This section is devoted to recalling the definition of the de Rham period map. We make definitions for general smooth $K$-schemes as far as possible, restricting to hyperbolic curves only when necessary. We will use the word ``connection'' throughout to refer to a vector bundle equipped with a connection.

Let $Z/K$ be a scheme over a field of characteristic zero. Let $D$ be a divisor of $Z$ with normal crossing singularities, and set $Z^{\prime}=Z-D$. Let $\mathrm{UnFC}(Z^{\prime})$ denote the Tannakian category of unipotent flat connections over $Z^{\prime}$: that is, the objects are connections that may be constructed as successive extensions of the trivial connection. Let $\mathrm{UnFC}(Z)(\log D)$ denote the category of flat unipotent logarithmic connections on $Z$. These are the flat meromorphic connections on $Z$ that are holomorphic outside $D$ and have at worst logarithmic singularities within $D$. Deligne has shown that every connection $\nabla$ over in $\mathrm{UnFC}(Z^{\prime})$ with a meromorphic extension to $Z$ has a canonical extension $\overline{\nabla}$ to a connection in $\mathrm{UnFC}(Z)(\log D)$. Furthermore, for $P\in D$, any such logarithmic connection on $C$ induces a connection, called the \emph{residue connection}, on $T_{P} Z\setminus\{0\}$. 

For any $K$-schemes $S$, let $Z_S$ denote the base-change of $Z$ to $S$, and write $\mathrm{Vect}_{S}$ for the category of vector bundles on $S$.  Let $b$ be a $K$-rational basepoint of $Z^{\prime}$, meaning that $b\in Z^{\prime}(K)$, or a tangential basepoint at $P \in D$, in which case $b\in (T_{P}Z)(K)$. When $b$ is a rational basepoint, there is an obvious fibre functor $F_b(S): \mathrm{UnFC}(Z^{\prime}_S)\rightarrow \mathrm{Vect}_S$ defined by sending a connection $\mathscr{E}$ to its fibre over $b$. When $b$ is a tangential basepoint at $P$, we define the fibre functor $F_b(S): \mathrm{UnFC}(Z_S)(\log(D))\rightarrow \mathrm{Vect}_{S}$ by sending a connection on $Z_S$ to the fibre of the residue connection $\mathrm{Res}\left(\mathscr{E}\right)$ on $T_{P} Z_S-\{0\}$ over $b$.

When $b$ is a rational basepoint, a \emph{pointed} unipotent flat connection on $Z^{\prime}$ (based at $b$) is a connection $(\mathscr{E},\nabla)$ in $\mathrm{UnFC}(Z^{\prime})$ together with an element $v \in \mathscr{E}_{b}$. When $b$ is a tangential basepoint, a pointed unipotent flat connection is a connection $(\mathscr{E},\nabla)$ in $\mathrm{UnFC}(Z)(\log(D))$ together with an element $v \in \mathrm{Res}\left(\mathscr{E}\right)_{b}$. In either case there exists a \emph{universal} unipotent pointed flat connection $(\mathscr{E}_0,\nabla_0,v_0)$: that is, for any other unipotent pointed flat connection $(\mathscr{E},\nabla,v)$, there exists a unique isomorphism of connections $\phi$ such that the underlying isomorphism $\mathscr{E}_{0} \rightarrow \mathscr{E}$ of $K$-vector bundles on $X$ satisfies $\phi_b(v_0)=v$. The universal pointed unipotent flat connection is unique up to a unique isomorphism.

The \emph{de Rham fundamental group} of $Z^{\prime}$ based at $b$, denoted $U^{dR}(Z^{\prime},b)$ (or simply $U$, though the dependence on $b$ is crucial), is defined to be the $K$-group scheme representing the functor $S\mapsto\mathrm{Aut}^{\otimes}F_b(S)$ of tensor-compatible automorphisms. Similarly, for $z^{\prime}\in Z^{\prime}$, the functor $S\mapsto \mathrm{Isom}^{\otimes}(F_b(S),F_{z^{\prime}}(S))$ of tensor-compatible isomorphisms is represented by a right $U^{dR}(Z^{\prime},b)$-torsor, denoted by $P^{dR}(Z^{\prime},b)$, which we call the \emph{de Rham path torsor}. On general grounds, $U^{dR}$ and $P^{dR}$ admit a certain filtration, called the \emph{Hodge filtration}: but we will see below that a lemma of Hadian provides a characterisation of the Hodge filtration which is rather more useful in practice. It is well-known that $U$ is a prounipotent group, so it is representable as the inverse limit of its descending central series quotients $U_n$, defined by \begin{equation*}
U^1=U,\quad U^n=[U,U^{n-1}],\quad U_n=U/U^{n+1}.
\end{equation*} The construction extends to produce descending central series quotients of its Lie algebra $\mathfrak{u}$, the universal enveloping algebra $\mathscr{U}(\mathfrak{u})$ of its Lie algebra, the de Rham path torsor, and unipotent connections on $Z^{\prime}$. We even have descending central series quotients for all of the quotients of these objects induced by quotients of $U$. The $n$th level quotient by the descending central series of any of these objects is always denoted with a subscript index $n$.

A logarithmic filtration $F^{\bullet}$ on a logarithmic connection $(\mathscr{E},\nabla) \in \mathrm{UnFC}(Z)(\log D)$ with poles along $D$ is a filtration of $\mathscr{E}$ by subbundles $F^i\mathscr{E}$ such that \begin{equation*}
    \mathscr{E}=F^m\mathscr{E}\supseteq F^{m+1}\mathscr{E}\supseteq\dots\supseteq F^n\mathscr{E}=0
\end{equation*} for some $m\leq n$ which satisfy Griffiths transversality: \begin{equation*}
\nabla\left(F^i\mathscr{E}\right)\subseteq F^{i-1}\mathscr{E}\otimes\Omega_{Z/K}^{1}\left(\log(D)\right).
\end{equation*} When the connection in question is pointed, with a marked element $v \in \mathscr{E}_{b}$, we additionally require that $v \in \left(F^{0}\mathscr{E}\right)_{b}$. The logarithmic extension of the universal connection inherits a logarithmic filtration, also called the \emph{Hodge filtration}, from the one on $U^{dR}$ mentioned earlier.

From now on, suppose that $K$ is a nonarchimedean field with residue field $k$, and that $Z^{\prime}$ has good reduction over $k$. Let $Z^{\prime}_k$ be the reduction of $Z^{\prime}$ to $k$ and write $\mathrm{UnIsoc}(Z^{\prime}_k)$ for the category of overconvergent unipotent isocrystals on $Z^{\prime}_k$. For $c \in Z^{\prime}_k(k)$, let $]c[$ denote the residue disk of $c$. The fibre functor $e_c: (\mathscr{E},\nabla)\mapsto \mathrm{ker}(\nabla)_{]c[}$, taking an element of $\mathrm{UnIsoc}(Z^{\prime}_k)$ to the $k$-vector space of its flat sections over $]c[$, provides $\mathrm{UnIsoc}(Z^{\prime}_k)$ with a Tannakian structure. The crystalline fundamental group based at $c$, denoted $U^{\mathrm{cr}}(Z^{\prime}_k,c)$, is the group scheme over the fraction field of the ring of Witt vectors of $k$ representing the tensor-compatible automorphisms, and the crystalline path torsor is the right $U^{\mathrm{cr}}(Z^{\prime}_k,c)$-torsor representing the tensor-compatible isomorphisms. If $c$ is instead a tangential basepoint, then there are suitable parallel definitions due to Besser--Furusho \cite{besserfurusho}. By a comparison theorem of Chiarellotto between crystalline and de Rham fundamental groups, the Frobenius automorphism on $U^{cr}$ induces a Frobenius automorphism on $P^{dR}(Z^{\prime},z^{\prime})$ for any $z^{\prime} \in Z^{\prime}$.

When $z^{\prime}, b \in Z^{\prime}(K)$, Besser \cite{besser} shows that there is a unique Frobenius-invariant de Rham path $p^{cr}(z^{\prime})$ from $b$ to $z^{\prime}$. When $z^{\prime}$ and $b$ are tangential basepoints, Besser--Furusho \cite{besserfurusho} prove that the existence of a unique Frobenius invariant de Rham path generally depends on a choice of a branch of the $p$-adic logarithm. However, if $z^{\prime}$ is an integral point, then a choice of branch of the $p$-adic logarithm is not required.

Kim defines \cite{unipalb} an \emph{admissible} right $U^{dR}$-torsor $T$ over $Z^{\prime}$ to be a right $U^{dR}$-torsor $T$ over $Z^{\prime}$ with a number of extra structures: an Eilbenberg--Maclane filtration on $\mathcal{O}(T)$, a Hodge filtration $F^{\bullet}$ on $\mathcal{O}(T)$ and a Frobenius automorphism $\phi:T \rightarrow T$ of $Z^{\prime}$-schemes. These structures are required to be compatible in various ways (with each other and with the Hodge filtration on $U^{dR}$), and there must be a $\phi$-invariant element $p^{cr}_{T}\in T(Z^{\prime})$ and a Hodge element $p^H_{T}\in F^0 T(Z^{\prime})$. When $Z^{\prime}$ is affine, there is a unique $p^{cr}_{T}$: otherwise, we may remove a point $P_0$ and make use of the functoriality of the Hodge filtration. Thus, once a point $p^{H}_{T}$ is chosen, there is a unique $u_T \in U^{dR}(Z^{\prime})$ such that $p^{cr}_{T}=p^{H}_{T} u_T$. We write $[T]$ for the image of $u_T$ inside $F^0 U^{dR}(Z^{\prime})\backslash U^{dR}(Z^{\prime})$: it is independent of the choice of $p^H_{T}$. \begin{prop}
    For a $K$-scheme $Z^{\prime}$, the map $T\mapsto [T]$ defines a bijection from the space of isomorphism classes of admissible right $U^{dR}$-torsors over $Z^{\prime}$ to the affine $K$-space $F^0 U^{dR}\backslash U^{dR}$.
\end{prop}

For each $z^{\prime} \in Z^{\prime}$, Kim has shown (in loc. cit.) that $P^{dR}(z^{\prime})$ is an admissible right $U^{dR}$-torsor over $Z^{\prime}$. Finally we may define the de Rham period map.
\begin{definition}
Let $b$ be a rational or tangential basepoint for $(Z, D, Z^{\prime})$. Then the de Rham period map is given by \begin{equation*}
    j^{dR}(Z^{\prime},b): x \mapsto [P^{dR}(Z^{\prime},b)] \in F^0 U^{dR}\backslash U^{dR},
\end{equation*} associating a point $z^{\prime} \in Z^{\prime}(K)$ to the class of the de Rham path torsor inside the moduli space of admissible $U^{dR}$-torsors.
\end{definition} In practice (for example, when working with Selmer schemes), we work with the composition of the de Rham period map with a Lie algebra-valued coordinate $\iota:  F^0 U^{dR}\backslash U^{dR}=\exp\left(F^0\mathfrak{u}\right)\backslash\exp\left(\mathfrak{u}\right)\rightarrow \mathfrak{u}$. At this point, we replace $(Z,D, Z^{\prime})$ with the data $(C, D, X)$ of a hyperbolic curve from the introduction, and explain how to work out $\iota\circ j^{dR}(X,x)$ for each $x \in X(K)$.

First we need to construct unipotent connections on $X$, and for this it is necessary to choose some bases. Fix some $P_0 \in C$ such that $P_0 \in D$ if $D\neq \varnothing$. Select a basis $\alpha_1,\dots,\alpha_g$ of global algebraic differential forms for $H^1_{dR}(C/K)$ and a basis $\beta_1,\dots,\beta_g$ of meromorphic differential forms for $H^1_{dR}(C/K)$ which are algebraic outside $P_0$. Choose meromorphic differential forms $\gamma_1,\dots,\gamma_{r-1}$ with poles contained in $D$ such that the $\alpha$s, $\beta$s and $\gamma$s together form a basis of $H^1_{dR}(X/K)$. Then if $X$ is affine, $\mathscr{U}(\mathfrak{u})$ is isomorphic to \begin{equation*}
    H=K\langle\langle A_1,B_1\dots, A_g, B_g,C_1,\dots C_{r-1}\rangle\rangle,
\end{equation*} the free cocommutative Hopf algebra on the symbols $A_i$, $B_j$, $C_k$, and one can arrange things so that the canonical isomorphisms $H_{1}^{dR}(X/K)\cong \log(U_1)\cong U_1$ identify $\alpha_i$ with $\log(A_i)$, $\beta_j$ with $\log(B_j)$ and $\gamma_k$ with $\log(C_k)$. In fact, we can identify $U$ with the grouplike elements of $H$ and $\mathfrak{u}$ with the primitive elements. If $D=\varnothing$, so that $X$ is a complete projective curve of genus at least $2$, then one obtains results analogous to the affine case by quotienting by the relation $\sum_{i}[A_i,B_i]=0$. Returning to the affine case, Kim has proved (loc. cit.) that the connection \begin{equation*}
\nabla_0=d-\sum_{i=1}^{g}\left(A_i\alpha_i+B_i\beta_i\right)-\sum_{k=1}^{r-1}\gamma_k C_k
\end{equation*} on $\mathcal{O}_{X}\otimes H$, which we call the \emph{naive connection}, represents the universal pointed unipotent connection on $X$. Fix any representative $(\mathcal{O}_{X}\otimes H, \nabla)$ of the universal unipotent pointed connection with marked element $1 \in H$ (which we interpret as the fibre over $b$). We choose a primitive $p^{H}_{\nabla}(x)$ in the $K$-algebra generated by $\left(F^0 \nabla\right)_{x}$ and
define $\iota_{\nabla}(m)=\log\left(\exp\left(-p_{\nabla}^{H}(x)\right)m\right)$, so that \begin{equation}\label{eq:logdR}
    \iota_{\nabla}\circ j^{dR}_{\nabla}(x)=\log\left[\exp\left(-p_{\nabla}^{H}(x)\right)p^{cr}_{\nabla}(x)\right].
\end{equation}
If $\nabla$ is any trivialised representative of the universal unipotent pointed connection on $X$ with marked element $1 \in H$ in the fibre over $b$, and $b \in X(K)$ is a rational basepoint, then $p_{\nabla}^{cr}(x)$ is the value at $x$ of the unique flat section $G$ of $\nabla$ which begins $G=1+$ terms of degree $\geq 1$ in $A_i$, $B_j$ and $C_k$. If $b \in (T_{P}C)(K)$ is a tangential basepoint, let $Y$ be an open set containing $P$ over which the logarithmic extension of $\nabla$ is trivialised, and choose some $y \in X(K)\cap Y(K)$: then for $x \in X(K)- Y(K)$, \begin{equation*}
p_{\nabla}^{cr}(x)=g_{\nabla}^{-1}(y)p_{\nabla^{\prime}}^{cr}(y)+p_{\nabla}^{cr}(x),
\end{equation*} where $g_{\nabla}$ is the gauge transformation over $X\cap Y$ gluing $\nabla$ to the logarithmic connection $\nabla^{\prime}$ on $Y$.

We turn now to the problem of determining $p^{H}_{\nabla}$. A lemma of Hadian \cite[Lemma 3.6]{hadian} states that there is a filtration $F^{\bullet}$ of logarithmic subbundles (with poles contained in $D$) on the logarithmic extension $\overline{\nabla_{0}}$ of $(\mathcal{O}_{X}\otimes\mathscr{U}(\mathfrak{u}),\nabla_0)$ to $C$ such that the exact sequence \begin{equation*}
    0\rightarrow V^{\otimes n}_{dR}\otimes \mathcal{O}_{C} \rightarrow \left(\overline{\nabla_0}\right)_{n}\rightarrow \left(\overline{\nabla_0}\right)_{n-1}\rightarrow 0
\end{equation*} of pointed logarithmic connections becomes an exact sequence of filtered pointed logarithmic connections. Here, we write $V_{dR}$ for $H_{1}^{dR}(X/K)$ and freely identify it with $U_1$. 

The Hodge filtration on $V^{\otimes n}_{dR}\otimes \mathcal{O}_{C}$ is the tensor product filtration of the filtration on $V_{dR}$ inducing by the de Rham complex of $X/K$ and the trivial filtration on $\mathcal{O}_{C}$. Specifically, \begin{equation}\label{eq:HodgefiltAG}
   F^{-j}\left(V^{\otimes n}_{dR}\otimes \mathcal{O}_{C}\right)=\begin{cases}
        \mathrm{Span}_{\mathcal{O}_{C}}\left\{\text{words }w \text{ in }A_i\text{s, }B_j\text{s, s.t. }\sum_{i=1}^{g}\mathrm{deg}_{A_i}w\leq |j|\right\} & j\leq 0\\
        0 & j>0.
    \end{cases}
\end{equation}
Hadian's pointed filtration is unique up to a unique isomorphism, and is completely determined by a choice of representative for the isomorphism class of the universal pointed connection. Thus, the Hodge filtration on any representative of the universal unipotent pointed connection is the restriction of this filtration from the logarithmic extension. Based on this observation, Beacom \cite[Section 4.2]{beacom} provides an algorithm for recursively computing the Hodge filtration on $\left(\nabla_{0}\right)_n$ from the Hodge filtration on $\left(\nabla_{0}\right)_{n-1}$ for affine elliptic and hyperelliptic curves.

Thus, given a representative $\nabla$ for the universal unipotent pointed connection, the data of the logarithm of its fundamental solution together with its Hodge filtration $F^{\bullet}$ is sufficient to define Kim's de Rham period map. For the naive connection $\nabla_{0}$ on $X$, the fundamental solution is given in terms of iterated integrals as\begin{equation*}
    G(z)=1+\sum_{\substack{\text{words }w \\ \text{ in }A_i, B_j, C_k}}G_w(z) w, \quad G_{A_i B_{k}^{2} C_j}=\int_{b}^{z} \alpha_i \beta_k \beta_{k} \gamma_j\text{ etc},
\end{equation*} so the computation of the de Rham period map with respect to $\nabla_0$ reduces to the computation of the Hodge filtration, at least when $b$ is a rational basepoint. When $b$ is a tangential basepoint, more work is necessary, but when the endpoint $z$ is an integral point contained in the open set containing $b$, the period map is obtained simply by replacing $\nabla_0$ with its logarithmic extension over that open set. In particular there is no need to choose a branch of the $p$-adic logarithm. Beacom's recursive algorithm for computing the Hodge filtration on descending central series quotients gives rise to a recursive algorithm \cite[Sections 5 and 6]{beacom} for computing $\left(j^{dR}_{\nabla_0}\right)_n$ from $\left(j^{dR}_{\nabla_0}\right)_{n-1}$. Unfortunately, it it appears to be very difficult to use his algorithm to directly write down the full Hodge filtration on $U$, or indeed on any infinite dimensional quotient of $U$, so a straightforward description of the de Rham period map has proved elusive.

However, any choice of representative for the universal unipotent connection, giving rise to a coordinate $\iota_{\nabla}$ on the period space, is sufficient for applications to Selmer schemes. As we shall demonstrate later, it is quite possible to compute explicitly the logarithm of the crystalline parallel transport operator $p^{cr}_{\nabla}$ (or more truthfully, its image in a certain large quotient of $U$) for any connection $\nabla$. Thus, it is natural to seek representatives of the universal connection whose Hodge filtration is as simple as possible. We will demonstrate that when $C$ is an elliptic curve and $D=\{\infty\}$, the elliptic KZB connection, introduced by Levin--Racinet \cite{levinracinet} and developed further by Luo \cite{luo} meets this requirement.

\section{The KZB connection and the universal connection}
From now on, $C=E$ is an elliptic curve over a nonarchimedean field $K$, with good reduction over the residue field $k$ of $K$, and $X$ is the affine curve $E-\{\infty\}$. Fix an embedding $K\hookrightarrow \mathbb{C}$. 

Without loss of generality, $X$ is given by the equation \begin{equation*}
y^2=4x^3-60 e_4 x-140 e_6,
\end{equation*} for some $e_4, e_6 \in K$. Let $\alpha=dx/y$ and $\beta=xdx/y$ so that $\{\alpha,\beta\}$ is a $K$-basis for $H^1_{dR}(X/K)$ and $\alpha$ is the restriction to $X$ of a global algebraic $1$-form on $E$. Let $\wp(z)$ denote the Weierstrass elliptic function associated to $E$, so that the Abel--Jacobi map is $AJ: P\mapsto z=\int_{\infty}^{P}\alpha$ and $x=AJ^{*}\wp(z)$. We have \begin{equation*}
\wp(z)=z^{-2}+\sum_{n=1}^{\infty}(2n+1)e_{2n+2}z^{2n},
\end{equation*} for some $e_k \in K$ and we set \begin{equation*}
    \wp_{2}(z)=\wp(z),\quad \wp_{k}(z)=\frac{(-1)^k}{(k-1)!}\frac{d^{k-2}\wp}{d z^{k-2}}
\end{equation*} for $k\geq 3$. The principal part of each $\wp_{k}(z)$ near $z=0$ is given by \begin{equation*}
    \wp_{k}(z)=z^{-k}+e_k+O(z),
\end{equation*} and the pullback functions \begin{equation*}
    P_k=AJ^{*}\left( \wp_{k}(z)-e_k\right)
\end{equation*} are algebraic on $X$. Following Weil \cite[Chapter II, Sections 7 and 8, and Chapter VI, Sections 1, 2 and 6]{weil}, we can characterise these by recurrence relations. Set $e_3=0$ and for $m\geq 5$, define $e_m$ by \begin{equation*}
\frac{1}{3}(m-3)(4m^2-1)e_m=\sum_{r=2}^{m-2}(2r-1)(2m-r-1)e_{2r}e_{2m-2r}.
\end{equation*} Then $P_2=x$, $P_3=-\frac{1}{2}y$ and for $m,n\geq 2$ we have \begin{equation*}
P_m P_n-P_{m+n}=(-1)^n \sum_{k=1}^{m-2}\binom{n+k-1}{k}e_{n+k}P_{m-k}+(-1)^m \sum_{k=1}^{n-2}\binom{m+k-1}{k}e_{m+k}P_{n-k}+(-1)^m\binom{m+n}{m}e_{m+n}.
\end{equation*} Observe that the undefined $P_1$ never actually contributes to the right hand side since it is always multiplied by an $e_k$ for $k$ odd, which is zero. 

Following Luo \cite[Sections 7 and 8]{luo}, we define the (algebraic) KZB connection by constructing suitable logarithmic connections on an open cover of $E$, then gluing the connections along a gauge transformation.
\begin{prop}
Fix some meromorphic function $f$ on $E$ such that $df-\beta$ is holomorphic at $\infty$, with $f=z^{-1}+O(z)$ with respect to the parameter $z$ above. Let $Y$ denote the open subset of $E$ obtained by removing the points at which $f$ has poles, except $\infty$. Define a connection \begin{equation}\label{eq:algkzbexpform}
    \nabla_{KZB}=d-\omega_{KZB},\quad \omega_{KZB}=\beta B+\alpha\exp\left(-\sum_{k=2}^{\infty}\frac{(-1)^k P_{k}\mathrm{ad}_{B}^{k}}{k}\right)A
\end{equation} on $X$ and a connection \begin{align*}
\nabla_{KZB}^{\prime}=d-\omega_{KZB}^{\prime},\quad
    \omega_{KZB}^{\prime}=(\beta-df)B+\alpha\exp\left(-f-\sum_{k=2}^{\infty}\frac{(-1)^k P_{k} \mathrm{ad}_{B}^{k}}{k}\right)A
\end{align*} on $Y$. The gauge transformation \begin{equation*}
    g=\exp(-fB),
\end{equation*} acting by $g(d+\omega)=d-dg.g^{-1}-g\omega g^{-1}$, satisfies $g(\nabla_{KZB})=\nabla_{KZB}^{\prime}$. Observe that $\nabla_{KZB}$ is nonsingular on $X$, and $\nabla_{KZB}^{\prime}$ is nonsingular on $Y$, apart from at $\infty$, where we claim it has a simple pole of residue $[B,A]$. It follows that $\nabla_{KZB}$, $\nabla_{KZB}^{\prime}$ and $g$ descend to an algebraic logarithmic $\mathscr{U}(\mathfrak{u})$-bundle with connection on $E$ with a simple pole at $\infty$. We denote this connection $(\nabla_{KZB},\nabla_{KZB}^{\prime},g)$.
\end{prop}
\begin{proof}
    It is straightforward to verify that the gauge transformation takes $\nabla_{KZB}$ to $\nabla_{KZB}^{\prime}$ using the identity $\mathrm{Ad}_{\exp(X)}=\exp\left(\mathrm{ad}(X)\right)$. To see that $\nabla_{KZB}^{\prime}$ has a regular singularity at $z=0$, recall that $P_{k}(z)$ has Laurent tail $z^{-k}+O(z)$ by construction, so \begin{align*}
        \exp\left(-f-\sum_{k=2}^{\infty}\frac{(-1)^k P_{k}\mathrm{ad}_{B}^{k}}{k}\right)&=\exp\left(-\sum_{k=1}^{\infty}\frac{(-1)^k\mathrm{ad}_{B}^{k}}{kz^k}+H(z)\right)=\left(\frac{\mathrm{ad}_{B}}{z}+1\right)\exp(H(z)),
    \end{align*} where $H$ is holomorphic at $z=0$. The claim concerning the algebraicity follows from the definition of the $P_k$.
\end{proof}
It will be convenient to write the KZB connections in the form \begin{gather}\label{eq:algKZB}
    \nabla_{KZB}=d-\beta B-\alpha A-\alpha\sum_{k=2}^{\infty}p_{k}\mathrm{ad}_{B}^{k}(A),\quad
    \nabla_{KZB}^{\prime}=d-\left(\beta-df \right)B-\alpha A-\alpha\sum_{k=1}^{\infty}q_{k}\mathrm{ad}_{B}^{k}(A),
\end{gather} where, following Luo \cite[Sections 7 and 8]{luo}, \begin{gather*}
    p_n=\sum_{2a_2+3a_3+\dots+na_n=n}\frac{1}{a_2! a_3!\dots a_n!}\prod_{k=2}^{n}\left(\frac{(-1)^{k+1}P_k}{k}\right)^{a_k}\\
    q_n=\sum_{a_1+2a_2+\dots+na_n=n}\frac{1}{a_1 !a_2 !\dots a_n!}\prod_{k=1}^{n}\left(\frac{(-1)^{k+1}P_k}{k}\right)^{a_k},
\end{gather*} with $P_1$ newly defined to be $-f$.
\begin{prop}\label{prop:kzbuniversal}
The connection $\nabla_{KZB}$ on $X$ with the marked element $1 \in \mathcal{O}_{X}\langle\langle A,B\rangle\rangle$ represents the universal unipotent pointed connection on $X$.
\end{prop}
\begin{proof}
It suffices to show that $\nabla_{KZB}$ and the naive connection are isomorphic as pointed connections on $X$. In fact, we will prove a stronger result: any connection on $\mathcal{O}_{X}\langle\langle A, B\rangle\rangle$ with marked element $1$ over $b$ of the form \begin{equation*}
    \nabla=d+\omega,\quad \omega=\alpha A+\beta B +\sum_{|w|>1}r_w w,\quad r_w\in \mathcal{O}_{X},
\end{equation*} where the sum is over words in $A$ and $B$ of total degree greater than one, is isomorphic to the naive connection. That is, we will show that it is possible to solve the equation \begin{equation}\label{eq:KZBisoNaive}
    dg=g\omega_0-\omega g,
\end{equation} where $\omega_0=\alpha A+\beta B$, for some $g \in \mathrm{Aut}\left(\mathcal{O}_{X}\langle\langle A, B\rangle\rangle\right)$ which satisfies $g(1)(b)=1$. To begin, we demonstrate that there exists such a gauge transformation which replaces the $r_w$ with $1$-forms belonging to $\mathrm{Span}_{K}\{\alpha,\beta\}$. Indeed, each $r_w$ is of the form $df_w+a_w\alpha+b_w\beta$ for some $f_w \in \mathcal{O}_{X}$ and $a_w, b_w \in K$. So we may rewrite $\omega$ as \begin{equation}
    \omega=\alpha (A+\sum_{w}a_w w)+ \beta (B+\sum_{w}b_w w) +\sum_{|w|=n}r_w w+\sum_{|w|>n}r_w w.
\end{equation} Then the invertible gauge transformation $g=1+\sum_{w}(f_w-f_w(b)) w$, acting by left multiplication, transforms $\omega$ into \begin{equation*}
\omega^{\prime}=\alpha (A+\sum_{w}a_w w)+ \beta (B+\sum_{w}b_w w)+\sum_{|w|=n}\left(a_w\alpha+b_w\beta\right)+\sum_{|w|>n}r^{\prime}_w w,
\end{equation*} and the demonstration is complete after induction. To finish the proof, it suffices to observe that the gauge transformation characterised by \begin{equation*}
    g(1)=1,\quad g(A)=A+\sum_{w}a_w w,\quad g(B)=B+\sum_{w}b_w w,
\end{equation*} and extended to $\mathscr{U}(\mathfrak{u})$ by multiplicativity and linearity, takes the naive connection with $\omega_0=\alpha A+\beta B$ to the connection with one-form \begin{equation*}
\omega^{\prime}=\alpha (A+\sum_{w}a_w w)+ \beta (B+\sum_{w}b_w w).
\end{equation*} To check that $g$ is invertible, fix the lexicographical ordering \begin{equation*}
    1<A<B<A^2<AB<BA<B^2<A^3<\dots
\end{equation*} on the standard basis $\mathscr{B}$ of $K\langle\langle A, B\rangle\rangle$ and observe that since $g(w)=w+\sum_{v > w \in \mathscr{B}}k_v v$ for some $k_v \in K$, the matrix for $g$ with respect to $\mathscr{B}$ is unipotent.
\end{proof}
We have shown that the universal unipotent connection on $X$ is represented by $\nabla_{KZB}$, so on general grounds, there is a Hodge filtration for $\nabla_{KZB}$. We demonstrate that the Hodge filtration has a particularly simple presentation.
\begin{prop}\label{prop:hodgefilt}
    The Hodge filtration $F^{\bullet}$ on $\nabla_{KZB}$ is \begin{equation*}
        F^{j}\mathcal{O}_{X}\langle\langle A,B\rangle\rangle=\mathrm{Span}_{\mathcal{O}_{X}}\left\{\text{words }w \text{ in }A, B \text{ s.t. }\mathrm{deg}_{A}w\leq -j\right\}.
    \end{equation*}
\end{prop}
\begin{proof}
    Define a filtration
    \begin{equation*}
        \left(F^{\prime}\right)^{j}\mathcal{O}_{Y}\langle\langle A,B\rangle\rangle=\mathrm{Span}_{\mathcal{O}_{Y}}\left\{\text{words }w \text{ in }A, B \text{ s.t. }\mathrm{deg}_{A}w\leq -j\right\}
    \end{equation*} on $\nabla_{KZB}^{\prime}$. It is clear that the gauge transformation $g$ induces an isomorphism of filtrations $F^{j}\cong\left(F^{\prime}\right)^{j}$, so the data $(F^{\bullet}, \left(F^{\prime}\right)^{\bullet}, g)$ is a candidate for the Hodge filtration on the logarithmic extension $\left(\nabla_{KZB},\nabla_{KZB}^{\prime},g\right)$ of $\nabla_{KZB}$. By Hadian's lemma, it suffices to check that $(F^{\bullet},\left(F^{\prime}\right)^{\bullet})$ descends to a filtration satisfying Griffiths transversality, and upgrades the exact sequence of logarithmic connections \begin{equation*}
        0 \rightarrow V_{dR}^{\otimes n}\otimes \mathcal{O}_{E}\rightarrow \left(\left(\nabla_{KZB}\right)_{n},\left(\nabla_{KZB}^{\prime}\right)_{n},g_n\right)\rightarrow \left(\left(\nabla_{KZB}\right)_{n-1},\left(\nabla_{KZB}^{\prime}\right)_{n-1},g_{n-1}\right)\rightarrow 0
    \end{equation*} to an exact sequence of filtered logarithmic connections. The latter is clear upon comparing the Hodge filtration on the logarithmic extension to the earlier definition (\ref{eq:HodgefiltAG}) of Hodge filtration on $V^{\otimes b}_{dR}\otimes \mathcal{O}_{E}$. The condition concerning Griffiths transversality amounts to the commutativity of the diagram \begin{center}
    \begin{tikzpicture}
    \node (A) at (-5.5,1) {$F^{i}\mathcal{O}_{X}\langle\langle A, B\rangle\rangle$};
    \node (B) at (1.5,1) {$F^{i}\mathcal{O}_{Y}\langle\langle A, B\rangle\rangle(\log(\infty))$};
    \node (C) at (-5.5,-1) {$F^{i-1}\mathcal{O}_{X}\langle\langle A, B\rangle\rangle \otimes\Omega_{X/K}^{1}$};
    \node (D) at (1.5,-1) {$F^{i-1}\mathcal{O}_{Y}\langle\langle A, B\rangle\rangle \otimes\Omega_{Y/K}^{1}\left(\log(\infty)\right)$};
    
    \draw[<->]  (A) -- (B);
    \draw[<->]  (C) -- (D) ;
    \node at (-2,-0.7) {$g$};
    \node at (-2,1.3) {$g$};
        \node at (2.3,0) {$\nabla_{KZB}^{\prime}$};
    \draw[right hook-latex]  (A) -- (C);
    \node at (-6.3,0) {$\nabla_{KZB}$};
    \draw[right hook-latex]  (B) -- (D);
 
    \end{tikzpicture}
\end{center} which follows from the fact that $g$ only increases the $B$-degree of elements of $K\langle\langle A, B\rangle\rangle$.
\end{proof}

\section{The de Rham period map}
At this stage, the only remaining obstacle to writing down the de Rham period map is the computation of the logarithm of a flat section of $\nabla_{KZB}$. Such a computation would require a basis for $\mathfrak{u}$ as a $K$-vector space, but recall that this is quite complicated: a basis for the free Lie algebra on a finite set of generators is given by Lyndon words. In order to obtain a reasonable description of the de Rham period map, we therefore concentrate on working out its image inside the maximal metabelian quotient $W=U/[[U,U],[U,U]]$ of $U$. This is in some sense the largest (Galois-equivariant) quotient of $U$ which is algebraically comprehensible. In fact, the set \begin{equation*}
    \left\{A, B, \sigma_{r,s}\right\}_{r,s\geq 0},\quad \sigma_{r,s}=\tau_{r,s}[A,B],\quad \tau_{r,s}=\mathrm{ad}_{B}^{r}\mathrm{ad}_{A}^{s}
\end{equation*} is a $K$-basis for the Lie algebra $\mathfrak{w}$ of $W$. The adjoint map on $\mathfrak{w}$ is injective, so there is an isomorphism $\mathfrak{w}\cong \mathrm{ad}(\mathfrak{w})\subseteq \mathrm{End}_{K}(\mathfrak{w})$, which induces an isomorphism $W\cong \mathrm{Ad}(W)\subseteq GL(W)$. If $G$ is a flat section of some connection $\nabla$ on $\mathcal{O}_{X}\otimes \mathscr{U}(\mathfrak{w})$ and $\nabla=d+\omega$ where $\omega$ acts as left multiplication by a primitive element, then $G$ is primitive. Furthermore, $\mathrm{Ad}(G)=\exp\left(\mathrm{ad}_{\log(G)}\right)$ is a flat section of the induced connection $\mathrm{Ad}(\nabla)=d+\mathrm{ad}_{\omega}$ on $\mathcal{O}_{X}\otimes \mathscr{U}(\mathrm{ad}(\mathfrak{w}))$. To describe such a section, we require a basis for a $K$-vector space containing $\exp\left(\mathrm{ad}(\mathfrak{w})\right)$ (not for all of $\mathscr{U}(\mathrm{ad}(\mathfrak{w}))$): \begin{equation*}
    \exp\left(\mathrm{ad}(\mathfrak{w})\right)\subset
     \mathrm{Span}_{K}\{\tau_{u,v},\mathrm{ad}_{\sigma_{r,s}}\}_{\substack{u,v\geq 0\\ r,s\geq 0}}.
\end{equation*} 
Thus, in order to understand $\log(G)$, it suffices to understand (part of) the flat section $\mathrm{Ad}(G)$ of the adjoint KZB connection $\mathrm{Ad}\left(\nabla_{KZB}\right)$ together with the exponential map on $\mathrm{ad}(\mathfrak{w})$. We prove a slightly more general result (the functions $p_1$ is not assumed to be zero) so that we can recycle it for the case where $b$ is a tangential basepoint.
\begin{prop}\label{prop:flatad}
    Let $\mathrm{Ad}\left(\nabla_{KZB}\right)=d-\mathrm{ad}_{\omega}$, where \begin{equation*}
        \mathrm{ad}_{\omega}=\beta \tau_{1,0}+\alpha \tau_{0,1}-\alpha\sum_{k=1}p_{k}\mathrm{ad}_{\sigma_{k-1,0}}.
    \end{equation*} Then the unique flat section of $\mathrm{Ad}\left(\nabla_{KZB}\right)$ of the form \begin{equation*}
        G=1+\sum_{\substack{u,v\geq 0\\ u+v>0}}G^{*}_{u,v}\tau_{u,v}+\sum_{r,s \geq 0}G_{r,s}\mathrm{ad}_{\sigma_{r,s}},\quad G \in \mathscr{U}(\mathrm{ad}(\mathfrak{w}))
    \end{equation*} satisfies \begin{equation}\label{eq:coeffsofflatstar}
        G^{*}_{1,0}=\int \beta,\quad G^{*}_{0,1}=\int \alpha,
    \end{equation} and $G_{r,s}$ is the iterated integral of the coefficient of $X^rY^s$ in \begin{equation}\label{eq:coeffsofflat}
        -\left(1-\alpha Y-\beta X\right)^{-1}\left(1+(1-\beta X)\sum_{i\geq 0}\sum_{k=1}^{i}\beta^{i-k+1}\alpha p_k X^i\right),
    \end{equation} where the ``iterated integral'' of an element of $\mathcal{O}_{X}\langle\langle\alpha, \beta \rangle\rangle$ is defined in the obvious way.
\end{prop}
\begin{proof}
   Comparing coefficients of $\tau_{1,0}$ and $\tau_{0,1}$ in $dG=\mathrm{ad}_\omega G$ yields (\ref{eq:coeffsofflatstar}), and an inspection of the coefficient of $\mathrm{ad}_{\sigma_{i,j}}$ in $\mathrm{ad}_\omega G$ yields \begin{equation*}
        dG_{i,j}=\begin{cases}
            -\alpha \left(p_1+\int\beta\right) & i=0, j=0\\
            \beta G_{i-1,0}-\alpha p_{i+1} & i\geq 1, j=0\\
            \alpha G_{0,j-1} & i=0, j\geq 1\\
            \beta G_{i-1,j}+\alpha G_{i,j-1} & i\geq 1, j\geq 1,
        \end{cases}
    \end{equation*} where we have used the identity \begin{equation*}
        \tau_{0,1}\tau_{u,v}=\begin{cases}
        \tau_{u,v+1} & (u,v)\neq (1,0)\\
        -\mathrm{ad}_{\sigma_{0,0}}-\tau_{1,1} & (u,v)=(1,0).
        \end{cases}
    \end{equation*} We obtain $G_{0,0}=-\int\alpha p_1-\int\alpha \beta$, then by induction we have \begin{equation*}
        G_{i,0}=\sum_{k=1}^{i}\int \beta^{i-k+1}\alpha p_k-\int \beta^{i}\alpha\beta, \quad G_{0,j}=-\int\alpha^{j+1}p_1-\int\alpha^{j+1}\beta.
    \end{equation*} If we let the $F_{i,j}$ be elements of the noncommutative ring $\mathcal{O}_{X}\langle\langle\alpha, \beta \rangle\rangle $ satisfying the recurrence relations \begin{equation*}
        F_{i,j}=\begin{cases}
            -\alpha \left(p_1+\beta\right) & i=0, j=0\\
            -\alpha p_{i+1}+\beta F_{i-1,0} & i\geq 1, j=0\\
            \alpha F_{0,j-1} & i=0, j\geq 1\\
            \beta F_{i-1,j}+\alpha F_{i,j-1} & i\geq 1, j\geq 1,
        \end{cases}
    \end{equation*} and the boundary conditions \begin{equation*}
        F_{i,0}=-\sum_{k=1}^{i}\beta^{i-k+1}\alpha p_k-\beta^i \alpha \beta,\quad F_{0,j}=-\alpha^{j+1} p_1-\alpha^{j+1}\beta,
    \end{equation*} then we find the following relation among generating series \begin{equation*}
        \sum_{i,j\geq 0}F_{i,j}X^{i}Y^{j}=-\left(1-\alpha Y-\beta X\right)^{-1}\left(1+(1-\beta X)\sum_{i\geq 0}\sum_{k=1}^{i}\beta^{i-k+1}\alpha p_k X^i\right),
    \end{equation*} which completes the proof.
\end{proof}
Next, we compute the logarithm of any grouplike element in $\mathscr{U}(\mathrm{ad}(\mathfrak{w}))$ (and we note that it is possible to do this for the maximal metabelian group on any finite number of generators).
\begin{prop}\label{prop:logarithm}
   Suppose that \begin{equation*}
    H=1+\sum_{u,v\geq 0}H^{*}_{u,v}\tau_{u,v}+\sum_{r,s\geq 0}H_{r,s}\mathrm{ad}_{\sigma_{r,s}}
   \end{equation*} is grouplike in $\mathscr{U}(\mathrm{ad}(\mathfrak{w}))$, with \begin{equation*}
       \log(H)=h_{B}\mathrm{ad}_{B}+h_{A}\mathrm{ad}_{A}+\sum_{r,s\geq 0}h_{r,s}\mathrm{ad}_{\sigma_{r,s}}.
   \end{equation*} Then we have $h_{A}=H^{*}_{0,1}$, $h_{B}=H^{*}_{1,0}$, and an equality of formal power series \begin{equation*}
       \sum_{r,s\geq 0}h_{r,s} X^s Y^r=\sum_{r,s\geq 0}\left(\sum_{\substack{u+n=r\\v+m=s}}\frac{B_{u+v}h_{B}^u h_{A}^v H_{r,s}}{u!v!}\right)X^s Y^r-h_A h_B T(h_A X, h_B Y),
   \end{equation*} where \begin{equation*}
       T(U,V)=\frac{1}{U}\left(1-\frac{U+V}{V}\frac{\exp(U)-1}{\exp(U+V)-1}\right)
   \end{equation*} and $B_{m}$ is the $m$th Bernoulli number from (\ref{eq:bernoulli}).
\end{prop} 
\begin{proof}
    An inspection of the coefficients of $\tau_{1,0}$ and $\tau_{0,1}$ in $\exp(\log(H))$ reveals the identities for $h_A$ and $h_B$. The rest of the claim is proved by studying the coefficient of $\mathrm{ad}_{\sigma_{r,s}}$. To begin, this yields \begin{align*}
        H_{r,s}&=\sum_{\ell=0}^{\infty}\frac{1}{\ell !}\left\{\text{coeff. of }\mathrm{ad}_{\sigma_{r,s}}\text{ in }\left(h_{A}\mathrm{ad}_{A}+h_{B}\mathrm{ad}_{B}+\sum_{r,s\geq 0}h_{r,s}\mathrm{ad}_{\sigma_{r,s}}\right)^{\ell}\right\}\\
        &=\sum_{\ell=0}^{\infty}\frac{1}{\ell !}\left\{\text{coeff. of }\mathrm{ad}_{\sigma_{r,s}}\text{ in }\left(\sum_{u,v\geq 0}h_{u,v}\mathrm{ad}_{\sigma_{u,v}}\right)^{\ell}+\sum_{\substack{u\leq r\\v \leq s}}h_{u,v}\times \text{ coeff. of }\tau_{r-u,s-v}\text{ in }\left(h_{A}\mathrm{ad}_{A}+h_{B}\mathrm{ad}_{B}\right)^{\ell-1}\right\}.
    \end{align*} But \begin{equation*}
        \left(h_{A}\mathrm{ad}_{A}+h_{B}\mathrm{ad}_{B}\right)^{\ell}=\sum_{i+j=\ell}h_{A}^{j}h_{B}^{i}\times\sum_{\substack{\text{words }w\\\text{in }A, B \text{ with }\\j A\text{s, }i B\text{s}}}\mathrm{ad}_{w},
    \end{equation*} where \begin{equation}\label{eq:adaverage}
        \sum_{\substack{\text{words }w\\\text{in }A, B \text{ with }\\j As, i Bs}}\mathrm{ad}_{w}=\binom{i+j}{i,j}\tau_{i,j}+\binom{i+j-1}{i-1,j}\mathrm{ad}_{\sigma_{i-1,j-1}}.
    \end{equation} We prove this via induction by writing \begin{align*}
        \sum_{\substack{\text{words }w\\\text{in }A, B \text{ with }\\j As, i Bs}}\mathrm{ad}_{w}&=\sum_{\substack{\text{words }w\\\text{in }A, B \text{ with }\\j-1 As, i Bs}}\mathrm{ad}_{w}\mathrm{ad}_{A}+\sum_{\substack{\text{words }w\\\text{in }A, B \text{ with }\\j As, i-1 Bs}}\mathrm{ad}_{w}\mathrm{ad}_{B},
    \end{align*} then substituting in (\ref{eq:adaverage}) and using relation $\tau_{i-1,j}\mathrm{ad}_{B}=\tau_{i,j}+\sigma_{i-1,j-1}$ for $i\geq 1$, with the convention that $\sigma_{i-1,j-1}=0$ when $j=0$.
    
    Then the coefficient of $\mathrm{ad}_{\sigma_{r,s}}$ in $\left(h_A \mathrm{ad}_{A}+h_{B}\mathrm{ad}_{B}\right)^{\ell}$ is \begin{equation*}
    \binom{r+s+1}{r,s+1}h_{A}^{s+1}h_{B}^{r+1}
    \end{equation*} if $r+s+2=\ell$, and is otherwise zero. Similarly, the coefficient of $\tau_{r-u,s-v}$ in $\left(h_A \mathrm{ad}_{A}+h_{B}\mathrm{ad}_{B}\right)^{\ell-1}$ is \begin{equation*}
        \binom{r-u+s-v}{r-u,s-v}h_{A}^{r-u}h_{B}^{s-v}
    \end{equation*} if $r-u+s-v=\ell-1$ and is otherwise zero. Therefore \begin{equation}\label{eq:pregenseries}
        H_{r,s}=\frac{h_{B}^{r+1}h_{A}^{s+1}}{(r+s+2)r! (s+1)!}+\sum_{\substack{0\leq u\leq r\\0\leq v \leq s}}\frac{h_{u,v}h_{B}^{r-u}h_{A}^{s-v}}{(r-u+s-v+1)(r-u)!(s-v)!}.
    \end{equation} We note that \begin{equation*}
        \sum_{r,s\geq 0}\frac{X^s Y^r}{(r+s+2)r!(s+1)!}=\frac{1}{X}\left(\frac{\exp(X+Y)-1}{X+Y}-\frac{\exp(X)-1}{X}\right)
    \end{equation*} because both sides are holomorphic functions near the origin which satisfy the differential equation \begin{equation*}
        \left(X\partial_X+Y\partial_Y+2\right)F=\frac{\exp(Y)(\exp(X)-1)}{X}.
    \end{equation*} The same reasoning applies to prove \begin{equation*}
        \sum_{r,s\geq 0}\frac{X^s Y^r}{(r+s+1)r!s!}=\frac{\exp(X+Y)-1}{X+Y}
    \end{equation*} where the differential equation is \begin{equation*}
    \left(X\partial_X+Y\partial Y+1\right)F=\exp(X+Y)
    \end{equation*}
    Then the proposition is proved upon passing to generating series in (\ref{eq:pregenseries})), and using the fact that \begin{equation*}
        \frac{\exp(X+Y)-1}{X+Y}=\sum_{r,s \geq 0}\frac{B_{r,s}}{r!s!}X^r Y^s.\qedhere
    \end{equation*}
\end{proof}
So we have proved that the function $\log(p^{cr}_{\nabla})(x)$ appearing in (\ref{eq:logdR}) is of the form \begin{equation*}
    \log(p^{cr}_{\nabla})(x)=\int_{b}^{x}\alpha A+\int_{b}^{x} \beta B+\sum_{r,s\geq 0}g_{r,s}\sigma_{r,s},
\end{equation*} where the $g_{r,s}$ are given in terms of the $G_{u,v}$ by Proposition \ref{prop:logarithm} and the $G_{u,v}$ are given in terms of the functions $p_n$ by (\ref{eq:coeffsofflat}). 

Now the proof of Theorem \ref{thm:periodmap} comes down to calculating the composition of $\iota_{KZB}$ with the period map.
\begin{proof}[Proof of Theorem \ref{thm:periodmap}]
   We begin by recalling a theorem of Kurlin \cite{kurlin}, which shows how the Baker--Campbell--Hausdorff formula simplifies for the maximal metabelian Lie algebra on two generators:\begin{equation}\label{eq:kurlinBCH}
      \log\left(e^Xe^Y\right)= X+Y+\frac{1}{\mathrm{ad}_{Y}}\left(1-\frac{\exp\left(\mathrm{ad}_{X}\right)-1}{\mathrm{ad}_{X}}\frac{\mathrm{ad}_{X}+\mathrm{ad}_{Y}}{\exp\left(\mathrm{ad}_{X}+\mathrm{ad}_{Y}\right)-1}\right)[X,Y].
   \end{equation} But $[\sigma_{r,s},[X,Y]]=0$ so when we set $X=-\int_{b}^{x}\beta B$ and $Y=\log(p_{\nabla}^{cr})(x)$, we can replace the $\mathrm{ad}_{Y}$s in (\ref{eq:kurlinBCH}) with $\int_{b}^{x}\alpha \mathrm{ad}_{A}+\int_{b}^{x}\beta \mathrm{ad}_{B}$. This yields the formula (\ref{eq:periodeq}).
   \end{proof}
    We conclude by describing the period map for a tangential basepoint $b$ at $\infty$. We use the notations $\alpha$, $\beta$, $e_n$ and $P_n$ from the introduction, and write $Y$ for the open subset of $E$ obtained by removing the points with $y=0$. Let $\beta^{'}$ be the one-form $\beta+dP_1$ and recall that at (\ref{eq:coordonperiod}) we constructed from the KZB connection a coordinate $\iota_{KZB}$ on the de Rham period space. We will assume that the endpoint of the de Rham period map is an integral point on $X$ so that we do not need to choose a branch of the $p$-adic logarithm. With respect to applications to Selmer schemes, this is by no means a severe restriction.
    \begin{thm}\label{thm:tangperiodmap}
        For $x \in X(K)\cap Y(K)$, we have \begin{equation*}
\iota_{KZB}\circ j^{dR}(X,b)(x)=\int_{b}^{x}\alpha A+\sum_{r,s\geq 0}g_{r,s}\sigma_{r,s}-\int_{b}^{x}\beta^{\prime} S\left(\int_{b}^{x}\alpha\mathrm{ad}_{A}, \int_{b}^{x}\beta^{\prime}\mathrm{ad}_{B}\right)\left(-\int_{b}^{x}\alpha \sigma_{0,0}+\sum_{r,s\geq 0}g_{r,s}\sigma_{r+1,s}\right)
   \end{equation*} where $j^{dR}(X,b)(x)$ is the de Rham period map and \begin{equation*}
       S(U,V)=\frac{1}{U+V}\left(1-\frac{\exp(-V)-1}{V}\frac{U+V}{\exp(U+V)-1}\right).
   \end{equation*} The $g_{r,s}$ are iterated integrals of certain algebraic functions on $E/K$, characterised by the equality of formal power series \begin{equation*}
       \sum_{r,s\geq 0}g_{r,s} X^s Y^r=\sum_{r,s\geq 0}\left(\sum_{\substack{u+n=r\\v+m=s}}\frac{B_{u+v}\left(\int_{b}^{x}\beta^{'}\right)^u\left(\int_{b}^{x}\alpha\right)^v G_{n,m}}{u!v!}\right)X^s Y^r-\left(\int_{b}^{x}\alpha\right) \left(\int_{b}^{x}\beta^{'}\right) T\left(\int_{b}^{x}\alpha X, \int_{b}^{x}\beta^{'} Y\right),
   \end{equation*} where \begin{equation*}
       T(U,V)=\frac{1}{U}\left(1-\frac{U+V}{V}\frac{\exp(U)-1}{\exp(U+V)-1}\right)
   \end{equation*} and $B_{m}$ is the $m$th Bernoulli number from (\ref{eq:bernoulli}). The algebraic iterated integrals $G_{r,s}$ are in turn characterised as the iterated integral of the coefficient of $X^r Y^s$ in 
\begin{equation*}
     -\left(1-\alpha Y-\beta^{'} X\right)^{-1}\left(1+(1-\beta^{'} X)\sum_{i\geq 0}\sum_{k=1}^{i}\left(\beta^{'}\right)^{i-k+1}\alpha q_k X^i\right),
\end{equation*} where we define the algebraic functions $q_k$ by \begin{equation*}
q_n=\sum_{a_1+2a_2+\dots+na_n=n}\frac{1}{a_1 !a_2 !\dots a_n!}\prod_{k=1}^{n}\left(\frac{(-1)^{k+1}P_k}{k}\right)^{a_k}.
    \end{equation*}
    \end{thm}
    \begin{proof}
        The proof is exactly the same as that of Theorem \ref{thm:periodmap}, except that we replace the KZB connection over $X$ throughout by the KZB connection over $Y$, as it appears at (\ref{eq:algKZB}).
    \end{proof}
\section*{Acknowledgements}
It is a great pleasure to express my gratitude to Minhyong Kim for over two years' worth of guidance on Selmer schemes and related topics, as well as his continued patience as the fruits of these discussions slowly coalesce into definitive results.

I have also benefited greatly from numerous discussions with Rob de Jeu and Herbert Gangl on the subject of de Rham period maps and polylogarithms. Special thanks are also due to Richard Hain and Federico Zerbini for clarifying some points about the KZB connection.

The author would like to thank the Isaac Newton Institute for Mathematical Sciences, Cambridge, for support and hospitality during the programme \emph{K-theory, algebraic cycles and motivic homotopy theory} where work on this paper was undertaken. This work was supported by EPSRC grant no EP/R014604/1.

The author would also to acknowledge the support of the University of Warwick through its Chancellor's International Scholarship. In addition, the author is grateful for the visiting studentship provided by the University of Edinburgh.

\printbibliography
\end{document}